\documentclass[11pt]{amsart}%
\usepackage{palatino, mathpazo}
\usepackage{amsfonts}
\usepackage{amsmath}
\usepackage{amssymb,latexsym}
\usepackage{graphicx}
\usepackage[mathscr]{eucal}
\usepackage{amssymb}%
\newtheorem{theorem}{Theorem}[section]

\newtheorem{corollary}[theorem]{Corollary}

\newtheorem{definition}[theorem]{Definition}
\newtheorem{lemma}[theorem]{Lemma}

\newtheorem{proposition}[theorem]{Proposition}
\theoremstyle{remark}

\newtheorem{example}[theorem]{Example}

\numberwithin{equation}{section}

\newcommand{\p}{\partial}

\newcommand{\T}{\vartheta}

\begin{document}
\title[Mean field equation and Lam\'{e} equation]{Geometric quantities arising from \\ bubbling analysis of mean field equations\\}

\author{Chang-Shou Lin}
\address{Taida Institute for Mathematical Sciences (TIMS), Center for Advanced Study in
Theoretical Sciences (CASTS), National Taiwan University, Taipei 10617, Taiwan }
\email{cslin@math.ntu.edu.tw}

\author{Chin-Lung Wang}
\address{Department of Mathematics, Taida Institute for Mathematical Sciences (TIMS),
National Taiwan University, Taipei 10617, Taiwan}
\email{dragon@math.ntu.edu.tw}

\begin{abstract} 
Let $E = \Bbb C/\Lambda$ be a flat torus and $G$ be its Green function with singularity at $0$. Consider the multiple Green function $G_n$ on $E^n$: 
\begin{equation*} 
G_{n}(z_1,\cdots,z_n) := \sum_{i < j} G(z_{i} - z_{j}) - n \sum_{i = 1}
^{n} G(z_{i}).
\end{equation*}
A critical point $a = (a_1, \cdots, a_n)$ of $G_n$ is called \emph{trivial} if $\{a_1, \cdots, a_n\} = \{-a_1, \cdots, -a_n\}$. For such a point $a$, two geometric quantities $D(a)$ and $H(a)$ arising from bubbling analysis of mean field equations are introduced. $D(a)$ is a global quantity measuring asymptotic expansion and $H(a)$ is the Hessian of $G_n$ at $a$. By way of geometry of Lam\'e curves developed in \cite{CLW}, we derive precise formulas to relate these two quantities. 
\end{abstract}
\maketitle

\section{Introduction}

Let $E = E_{\tau}:=\mathbb{C}/\Lambda_{\tau}$ be a flat torus where $\Lambda_{\tau}=\mathbb{Z+Z}\tau$ and $\tau \in \mathbb{H} = \left \{\tau \mid \operatorname{Im}\tau > 0\right \} $. We use the convention $\omega_1 = 1$, $\omega_2 = \tau$ and $\omega_3 = 1 + \tau$. Consider the following mean field equation with singular strength $\rho > 0$:
\begin{equation} \label{mfe}
\triangle u + e^{u} = \rho \,\delta_{0} \quad \text{in $E$}, 
\end{equation}
where $\delta_{0}$ is the Dirac measure at $0$. Solutions to this simple looking equation (\ref{mfe}) possess a rich structure from either the point of view of partial differential equations or of integrable systems. See \cite{CLW, CL-1, CLW4}.

Not surprisingly, \eqref{mfe} is related to various research areas. In conformal geometry, a solution $u(x)$ to \eqref{mfe} leads to a metric $ds^{2}=\frac{1}{2}e^{u}\,(dx^{2} + dy^2)$ with constant Gaussian curvature $+1$ acquiring a conic singularity at $0$. It also appears in statistical physics as the equation for the \emph{mean field limit} of the Euler flow in Onsager's vortex model, hence its name. In the physical model of superconductivity, (\ref{mfe}) is one of limiting equations of the well-known Chern--Simons--Higgs equation as the coupling parameter tends to $0$. We refer the interested readers to \cite{CY, CL3, Choe, LW, LY, NT1} and references therein for recent development on this equation.

One important feature of \eqref{mfe} is the so-called \emph{bubbling phenomena}. Let $u_{k}$ be a sequence of solutions to \eqref{mfe} with $\rho = \rho_{k} \rightarrow 8\pi n$, $n\in \mathbb{N}$, and $\max_{E} u_{k}(z)\rightarrow+\infty$ as $k\rightarrow+\infty$. Then it was proved in \cite{CL3} that $u_{k}$ has exactly $n$ blowup points
$\{a_{1}, \cdots, a_{n}\}$ in $E$ and $a_{i} \ne 0$ for all $i$. The well-known \emph{Pohozaev identity} says that the position of these blowup points are determined by the following system of equations:

\begin{equation} \label{poho}
n\nabla G(a_{i}) = \sum_{j\not =i}^{n}\nabla G(a_{i} - a_{j}), \qquad 1\leq
i\leq n. 
\end{equation}
Here $G(z, w)=G(z - w)$ is the Green function on $E$ defined by
\begin{equation*} 
\begin{cases}
\displaystyle -\triangle G = \delta_0 - \frac{1}{|E|} & \mbox{on $E$}, \\
\int_{E} G = 0,
\end{cases}
\end{equation*}
and $|E|$ is the area of $E$. 

\indent If $\rho_k = 8\pi n$ for all $k$, then $\{u_k\}$ consists of \emph{type II solutions} with explicit blowup behavior (cf.~\cite{CLW}). On the other hand, we have
\medskip

\noindent \textbf{Theorem A.} \cite{CLW, CL-1} {\it Let $u_{k}$ be a
sequence of bubbling solutions of equation \eqref{mfe} with $\rho =\rho
_{k}\rightarrow 8\pi n$, $n \in \mathbb{N}$. If $\rho_{k} \ne 8\pi n$ for large $k$, then

\begin{itemize}
\item[(1)] The blowup set $a=\{a_{1}, \cdots, a_{n}\}$ satisfies 
$$
\{a_{1}, \cdots, a_{n}\} = \{-a_{1}, \cdots, -a_{n} \} \quad \mbox{in $E$}.
$$

\item[(2)] Let $\lambda_{k} :=\max_{E} u_{k}(z)$, then there is a constant $D(a)$ such that
\begin{equation} \label{bubb}
\rho_{k} - 8\pi n = (D(a) + o(1)) e^{-\lambda_{k}}.
\end{equation}
\end{itemize}
}

From \eqref{bubb}, the quantity $D(a)$ plays a fundamental role in controling the sign of $\rho_{k} - 8\pi n$. Thus it provides one of the key geometric messages for bubbling solutions $u_{k}$. The question is how to compute $D(a)$? 

There exists a complicate expression for $D(a)$ which we will recall in \eqref{D-invariant} below. Define the regular part $\tilde{G}(z,w)$ of $G(z,w)$ by
\[
\tilde{G}(z,w) := G(z,w) + \frac{1}{2\pi} \log|z - w|.
\]
Given a blowup set $a = \{ a_{1}, \cdots ,a_{n} \}$ as in Theorem A (1), we set
\begin{equation} \label{f-q}
\begin{split}
f_{a_{i}}(z) &=  8\pi \bigg(\tilde{G}(z,a_{i}) - \tilde{G}(a_{i}, a_{i}) + \sum_{j \ne i}(G(z, a_{j}) - G(a_{i}, a_{j})) \\
& \qquad  -n(G(z) - G(a_{i}))\bigg),
\end{split}
\end{equation}
\begin{equation} \label{mu}
\mu_{i} :=\exp \bigg(8\pi(\tilde{G}(a_{i}, a_{i}) + \sum_{j \ne i}G(a_{i}, a_{j}) - n G(a_{i}))\bigg). 
\end{equation}
Then $D(a)$ can be calculated by
\begin{equation} \label{D-invariant}
D(a) = \lim_{r\rightarrow 0} \sum_{i = 1}^{n} \mu_{i} \left(  \int_{\Omega
_{i}\setminus B_{r}(a_{i})} \frac{e^{f_{a_{i}}(z)} - 1}{|z - a_{i}|^{4}} - \int_{\mathbb{R}^{2} \setminus \Omega_{i}} \frac{1}{|z - a_{i}|^{4}}\right),
\end{equation}
where $\Omega_{i}$ is any open neighborhood of $a_{i}$ in $E$ such that
$\Omega_{i}\cap \Omega_{j} = \emptyset$ for $i \ne j$ and $\bigcup_{i=1}^{n}\bar{\Omega}_{i} = E$. The limit exists since $f_{a_{i}}(z) = O(|z - a_{i}|^{3})$ plus a quadratic harmonic function for all $i$. For a proof, see \cite{LY}.

Consider the divisor (complete diagonal) in $(E^\times)^n$:
$$
\Delta_{n} = \{(z_{1},\cdots, z_{n}) \in (E^{\times})^{n} \mid \text{$z_{i} = z_{j}
$ for some $i \ne j$} \}
$$ 
and define the \emph{multiple Green function} $G_{n}(z) = G_{n}(z; \tau)$ on $(E^\times)^n \setminus \Delta_n$ by
\begin{equation} \label{multiple-G}
G_{n}(z) := \sum_{i < j} G(z_{i} - z_{j}) - n \sum_{i=1}^{n} G(z_{i}).
\end{equation}
Notice that $G_{n}$ is invariant under the permutation group $S_n$. It is clear that the system \eqref{poho} gives the critical point equations of $G_{n}$. A critical point $a$ is called \emph{trivial } if $\{a_{1}, \cdots, a_{n}\} = \{-a_{1}, \cdots , -a_{n}\}$ in $E$. Theorem A (1) says that the blowup set of a sequence of bubbling solutions $u_{k}$ of \eqref{mfe} with $\rho_{k} \ne 8\pi n$ for large $k$ is a trivial critical point of $G_{n}$.

To proceed, it is crucial and natural to ask when is a trivial critical point a degenerate critical point? To answer this question, we need to study the Hessian $H(a)$ at a trivial critical point $a$:
\begin{equation} \label{Hessian}
H(a) := \det D^{2} G_{n}(a).
\end{equation}
The quantity $H(a)$ can be used to determine the local maximum points of $u_k$ near $a_i$, $1\le i\le n$, and to provide other useful geometric information for the bubbling solutions $u_{k}$ (cf.~\cite{CL-1}). 

There are many potential applications of these two quantities. For example, $H(a)$ and $D(a)$ together imply \emph{local uniqueness} of bubbling solutions, as described in the following theorem:\medskip

\noindent \textbf{Theorem B.} \emph{Let $u_{k}(z)$ and $\tilde{u}_{k}(z)$ be two sequences of solutions to equation \eqref{mfe} with the same parameter $\rho_{k} \rightarrow 8\pi n$ and $\rho_{k} \ne 8\pi n$ for large $k$. If they have the same blowup set $a = \{a_{1}, \cdots, a_{n}\}$ and both $H(a)$ and $D(a)$ do not vanish, then $u_{k}(z)=\tilde{u}_{k}(z)$ for large $k$.}
\medskip

The proof of Theorem B will be given in a forthcoming paper by the first author. It is unexpected since after some suitable scaling at each blowup point
$a_{i}$, the solution $u_{k}(z)$\emph{ }(resp.\emph{ }$\tilde{u}_{k}(z)$)
converge to a solution of equation%
\begin{equation*} 
\Delta w + e^{w} = 0 \quad \text{in $\mathbb{R}^{2}$},\qquad \int_{\mathbb{R}^{2}} e^{w} < \infty,
\end{equation*}
and it is easy to see that the linearized operator $\Delta + e^{w}$ has non-trivial kernel. To prove the uniqueness, we have to overcome the difficulty caused by the degeneracy of the operator $\Delta + e^{w}$.

Surprisingly, these two quantities $D(a)$ and $H(a)$ are related
to each other as shown by the main result of this paper:

\begin{theorem} [=Theorem \ref{DD-conj}] \label{thm1} 
For fixed $n\in \mathbb{N}$ and any trivial critical point $a$ of $G_{n}(z)$, there exists $c_a \geq 0$ such that
\begin{equation} \label{ff1}
H(a) = (-1)^{n} c_a D(a). 
\end{equation}
Moreover, $c_a > 0$ if and only if $B_{a} := (2n - 1) \sum_{i = 1}^{n} \wp(a_{i})$
is not a multiple root of the Lam\'{e} polynomial $\ell_{n}(B)$.
\end{theorem}

Here is an outline of the proof, together with a brief description on the content of each section:

The mean field equation \eqref{mfe} is closely related to the \emph{Lam\'e equation} $y'' = (n(n + 1) \wp + B) y$. To prove \eqref{ff1}, a key step is to express $D(a)$ in terms of quantities at a branch point of the hyperelliptic curve $Y_n \to \Bbb C$ associated to the Lam\'e equation. This \emph{Lam\'e curve} $Y_n$ can be represented by $C^2 = \ell_n(B)$ where the \emph{Lam\'{e} polynomial} $\ell_n(B)$ has no multiple roots except for \emph{finitely many} isomorphic classes of tori. This theory is well developed in \cite{CLW} and the results we need will be reviewed in \S \ref{geom-Xn} (cf.~Theorem \ref{thm3}). 

In \S \ref{D-for-n} we study the quantity $D(a)$ in details and derive the above mentioned expression of $D(a)$ in Theorem \ref{D(p)-formula}. In fact, the Lam\'e curve encodes the $n - 1$ algebraic constraints of the system \eqref{poho}, with the remaining analytic constraint being $\sum_{i = 1}^n \nabla G(a_i) = 0$. It is thus natural to study the map $a \mapsto \phi(a) := -4\pi \sum_{i = 1}^n \nabla G(a_i)$ for $a \in Y_n$. It turns out that $D(a)$ is expressible in terms of the Jacobian of $\phi$ (Corollary \ref{D=Jac}).

The proof of Theorem \ref{thm1} is completed in \S \ref{s:proof} by a process called \emph{analytic adjunction}. The idea is simple: The quantity $H(a)$ is a (real) $2n$-dimensional Hessian on $E^n/S_n$ while $D(a)$ can be regarded as a two dimensional Hessian on $Y_n \subset E^n/S_n$. To relate $H(a)$ with $D(a)$ it amounts to reducing the determinant by substituting the $n - 1$ (complex) algebraic equations defining $Y_n$ into it. We end this paper by investigating the case $n = 2$ in Example \ref{DD2} where the value of $c_a$ (given in \eqref{cp-expression}) is seem in more explicit terms.

\section{Lam\'{e} equations and Lam\'e curves \cite{CLW}}
\label{geom-Xn}

Let $\wp(z) = \wp(z; \tau) $ be the Weierstrass elliptic function with periods $\Lambda_\tau$:
\begin{equation*}\label{40-3}
\wp(z; \tau) := \frac{1}{z^{2}}+\sum_{\omega \in \Lambda_{\tau}\setminus \left \{
0\right \}  }\left(  \frac{1}{(z-\omega)^{2}} - \frac{1}{\omega^{2}}\right)  
\end{equation*}
which satisfies the well known cubic equation 
\begin{equation*}\label{40-4}
\wp^{\prime}(z; \tau)^{2} = 4\wp(z; \tau)^{3} - g_{2}(\tau) \wp(z; \tau) - g_{3}(\tau).
\end{equation*}
Let $\zeta(z) = \zeta(z; \tau): = -\int^{z}\wp(\xi; \tau)\,d\xi$ be the Weierstrass
zeta function with quasi-periods $\eta_{1}(\tau)$ and $\eta_{2}(\tau)$:
\begin{equation*} \label{40-2}
\eta_{i}(\tau) := \zeta(z + \omega_i; \tau) - \zeta(z; \tau), \quad i = 1, 2,
\end{equation*}
and $\sigma(z) = \sigma(z; \tau)$ be the Weierstrass sigma function defined by
$\sigma(z) = \exp \int^z \zeta(\xi)\,d\xi$. $\sigma(z)$ is an odd entire function with simple zeros at $\Lambda_\tau$.

The Green function on $E$ can be expressed in terms of elliptic functions. In \cite{LW}, we proved that
\begin{equation}\label{G_z}
-4\pi \frac{\partial G}{\partial z}(z) = \zeta(z) - r\eta_{1} - s\eta_{2} = \zeta(z) - z\eta_1 + 2\pi i s,
\end{equation}
where $z = r + s \tau$ with $r, s\in \mathbb{R}$. Using \eqref{G_z}, equations \eqref{poho} can be translated into the following equivalent system: Consider $a = (a_1, \cdots, a_n) \in E^n$, subject to the constraint $a \in (E^\times)^n \setminus \Delta_n$, that is
\begin{equation} \label{fc5}
a_{i}\ne 0,\quad \mbox{$a_{i} \ne a_{j}$ for $i \ne j$}.
\end{equation}
Then
\begin{equation} \label{ff3}
\sum_{j \ne i} (\zeta(a_{i} - a_{j}) + \zeta(a_{j}) - \zeta(a_{i})) = 0, \qquad 1\leq i \leq n 
\end{equation}
(there are only $n - 1$ independent equations), and
\begin{equation} \label{ff4}
\sum_{i = 1}^{n} \nabla G(a_{i}) = 0.
\end{equation}

We will use \eqref{fc5}--\eqref{ff4} to connect a critical point of $G_n$ defined in \eqref{multiple-G} with the classical Lam\'{e} equation. For the reader's convenience, we review some basics on it and refer the readers to \cite{CLW, Poole, Whittaker-Watson} for further details. 

Recall the Lam\'{e} equation
\begin{equation} \label{Lame0}
\mathcal{L}_{n,B}: \qquad y''(z) = (n(n + 1) \wp(z) + B) y(z),
\end{equation}
where $n\in \mathbb{R}_{\geq-1/2}$ and $B\in \mathbb{C}$ are its \emph{index}
and \emph{accessory parameter} respectively. In general, a solution $y(z)$ is a multi-valued meromorphic function on $\mathbb{C}$ with branch points at $\Lambda$. Any lattice point is a regular singular point with local exponents $-n$ and $n+1$. In this paper we consider only $n\in \mathbb{N}$.

For $a = (a_{1},\cdots, a_{n})$, we consider the \emph{Hermite--Halphen ansatz}:
\begin{equation} \label{ff2-1}
y_{a}(z):=e^{z\sum_{i=1}^{n}\zeta(a_{i})}\frac{\prod_{i=1}^{n}\sigma(z-a_{i}%
)}{\sigma(z)^{n}}.
\end{equation}

\begin{theorem} [\cite{CLW, Whittaker-Watson}] 
Suppose that $a = (a_{1},\cdots, a_{n}) \in (E^\times)^n \setminus \Delta_n$. Then $y_{a}(z)$ is a solution to $\mathcal{L}_{n, B}$ for some $B$ if and only if $a$ satisfies \eqref{ff3} and 
\begin{equation} \label{b-a}
B = B_{a} := (2n - 1) \sum_{i=1}^{n}\wp(a_{i}).
\end{equation}
\end{theorem}

Note that if $a = (a_{1},\cdots, a_{n}) \in (E^\times)^n \setminus \Delta_n$ satisfies \eqref{ff3}, then so does $-a = (-a_{1}, \cdots, -a_{n})$, and then $y_{-a}(z)$ is also a solution of the same Lam\'{e} equation
because $B_{a}=B_{-a}$. Clearly $y_{a}(z)$ and $y_{-a}(z)$ are linearly independent if and only if $\{a_{1},\cdots, a_{n}\} \ne \{-a_{1},\cdots, -a_{n}\}$ in $E$. Furthermore, the condition actually implies that
\begin{equation} \label{fc6}
\{a_{1}, \cdots, a_{n}\} \cap \{-a_{1},\cdots, -a_{n}\} = \emptyset
\end{equation}
because $y_{a}(z)$ and $y_{-a}(z)$ can not have common zeros. For otherwise the Wronskian of $(y_{a}(z), y_{-a}(z))$ would be identically zero, which forces that $y_{a}(z)$, $y_{-a}(z)$ are linearly dependent.

\begin{definition}
Suppose that $a = (a_{1}, \cdots, a_{n}) \in (E^\times)^n \setminus \Delta_n$ satisfies \eqref{ff3}. Then $a$ is called a branch point if $\{a_{1}, \cdots, a_{n}\} = \{-a_{1},\cdots, -a_{n}\}$ in $E$.
\end{definition}

Note that if $a$ is \emph{not} a branch point, then $\wp(a_{i}) \ne \wp(a_{j})$ for
$i \ne j$. By the addition formula
\[
\zeta(u+v)-\zeta(u)-\zeta(v)=\frac{1}{2}\frac{\wp'(u) - \wp'(v)}{\wp(u) - \wp(v)},
\]
the system \eqref{ff3} is equivalent to
\begin{equation} \label{fc7}
\sum_{j \ne i} \frac{\wp'(a_{i}) + \wp'(a_{j})}{\wp(a_{i}) - \wp(a_{j})} = 0, \qquad 1\leq i\leq n.
\end{equation}

The following non-obvious equivalence is crucial for our purpose:

\begin{proposition} \cite[Proposition 5.8.3]{CLW} \label{proppp}  
Suppose that $a = (a_{1}, \cdots, a_{n}) \in (E^\times)^n$ satisfies $\wp(a_{i}) \ne \wp(a_{j})$ for $i \ne j$. Then \eqref{fc7} is equivalent to
\begin{equation} \label{fc8}
\sum_{i = 1}^{n} \wp'(a_{i}) \wp(a_{i})^{l} = 0, \qquad 0\leq l\leq n - 2.
\end{equation}
\end{proposition}


Let $a \in (E^\times)^n \setminus \Delta_n$ satisfy (\ref{ff3}) and suppose that it is not a branch point. Then \eqref{fc8} implies that
\begin{equation} \label{g=CP-1}
g_{a}(z):=\sum_{i=1}^{n}\frac{\wp'(a_{i})}{\wp(z) - \wp(a_{i})} =\frac{C(a)} {\prod_{i = 1}^{n}(\wp(z) - \wp(a_{i}))}
\end{equation}
for a constant $C(a)\neq0$. Equivalently,
\begin{equation} \label{Ca0}
C(a) = \sum_{i=1}^{n} \wp'(a_{i}) \prod_{j\neq i} (\wp(z) - \wp(a_{j})).
\end{equation}
There are various ways to represent $C(a)$ by plugging in different values of
$z$ in (\ref{Ca0}). For example, for $z=a_{i}$ we get
\begin{equation} \label{Ca-1}
C(a)=\wp^{\prime}(a_{i})\prod_{j\neq i}(\wp(a_{i})-\wp(a_{j}))
\end{equation}
which is independent of the choices of $i$. Notice that if $a$ is a branch point then $g_{a}(z)\equiv 0$ and so $C(a) = 0$.

Then we have the following important result:

\begin{theorem} \cite{CLW} \label{thm3} 
There exists a polynomial $\ell_n(B) = \ell_{n}(B; g_2, g_3) \in \mathbb{Q}[g_{2}, g_{3}][B]$ of degree $2n + 1$ in $B$ such that if $a \in (E^\times)^n \setminus \Delta_n$ satisfies \eqref{ff3}, then $C^{2} = \ell_{n}(B)$, where $C = C(a)$ and $B = B_{a}$ are given in \eqref{Ca-1} and \eqref{b-a} respectively.
\end{theorem}

This polynomial $\ell_{n}(B)$ is called the \emph{Lam\'{e} polynomial} in the
literature. 

Let $Y_n = Y_{n}(\tau) \subset {\rm Sym}^n\,E = E^n/S_n$ be the set of $a=\{a_{1},\cdots, a_{n}\}$ which satisfies \eqref{fc5} and \eqref{ff3}. Clearly $-a:=\{-a_{1},\cdots, -a_{n}\} \in Y_{n}$ if $a\in Y_{n}$, and $a\in Y_{n}$ is a \emph{branch point} if $a = -a$ in $E$. Then the map $B: Y_{n}\rightarrow \mathbb{C}$ in \eqref{b-a} is a ramified covering of degree $2$, and Theorem \ref{thm3} implies that
\[
Y_{n} \cong \left \{  (B, C) \mid C^{2} = \ell_{n}(B) \right \},
\]
(cf.~\cite[Theorem 7.4]{CLW}). Therefore, $Y_{n}$ is a hyperelliptic curve, known as the \emph{Lam\'e curve}. Furthermore, \emph{$Y_{n}$ is singular at a trivial critical point $a$ if and only if $B_{a}$ is a multiple zero of $\ell_n(B)$}. For later usage, we denote
$$
X_{n} := \{ a \in Y_{n} \mid \mbox{$a$ is not a branch point} \} \subset Y_n.
$$

Since $a$ is a branch point of $Y_n$ if and only if it is a trivial critical point of $G_n$. From now on we will switch these two notions freely.

There are several ways to compute the Lam\'{e} polynomial $\ell_{n}(B)$. A recursive construction can be found in \cite[Theorem 7.4]{CLW}.

\begin{example} \cite{Beukers-Waall, CLW} \label{lame-curve}
$\ell_n(B)$ for $n = 1, 2$. Denote $e_k=\wp(\frac{\omega_k}{2})$ for $k = 1, 2, 3$.

(1) $n = 1$, $\bar X_1 \cong E$, $C^2 = \ell_1(B) = 4 B^3 - g_2 B - g_3 = 4\prod_{i = 1}^3(B - e_i)$.

(2) $n = 2$ (notice that $e_1 + e_2 + e_3 = 0$), 
\begin{equation*}
\begin{split}
C^2 = \ell_2(B) &= \tfrac{4}{81} B^5 - \tfrac{7}{27} g_2 B^3 + \tfrac{1}{3} g_3 B^2 + \tfrac{1}{3} g_2^2 B - g_2 g_3 \\
&= \frac{2^2}{3^4}(B^2 - 3g_2) \prod_{i = 1}^3 (B + 3 e_i).
\end{split}
\end{equation*} 

Consequently, $\ell_2(B; \tau)$ has multiple zeros if and only if $g_2(\tau) = 0$, that is $\tau$ is equivalent to $e^{\pi i/3}$ under the ${\rm SL}(2, \mathbb{Z})$ action. 

If $a=\{a_1, a_2\}$ is a branch point of $Y_2$, then $\{a_1, a_2\} = \{-a_1, -a_2\}$ in $E$ implies that either (1) $a =\{\frac{1}{2}\omega_i,\frac{1}{2}\omega_j\}$ with $\{i,j,k\}=\{1,2,3\}$, which corresponds to $B_a=3(e_i+e_j)=-3e_k$, or (2) $a_1=-a_2\neq \frac{\omega_k}{2}$. Then $\pm\sqrt{3g_2} = B_a = 6\wp(a_1)$, i.e.~$\wp(a_1) = \pm\sqrt{g_2/12}$. We conclude that the branch points of $Y_2$ are given by $\{(\tfrac{1}{2} \omega_i, \tfrac{1}{2} \omega_j) \mid i \ne j\}$ and $\{(q_{\pm}, -q_{\pm}) \mid \wp(q_{\pm}) = \pm \sqrt{g_2/12}\}$.
\end{example}

\section{The invariant $D(a)$ and its geometric meaning} \label{D-for-n}
\setcounter{equation}{0}

The purpose of this section is to generalize the invariant $D(a)$ studied in \cite{LW4} for $\rho = 8\pi$, where $a$ is a half-period point, to the general case $\rho = 8\pi n$ for all $n \in \Bbb N$. $D(a)$ is fundamental in analyzing the bubbling behavior of a sequence $u_k$ with $\rho_k \to 8\pi n$. By Theorem A, the bubbling loci $a = \{a_1, \cdots, a_n\}$ must be a branch point of $Y_n$ if $\rho_k \ne 8\pi n$ for $k$ large. Thus it is essential to study the geometric meaning of $D(a)$ at those $2n + 1$ branch points as in the case $n = 1$ in \cite[Theorem 0.4]{LW4}.

For $a = (a_1, \cdots, a_n) \in (E^\times)^n \setminus \Delta_n$ a trivial critical point, we recall \eqref{D-invariant}:
\begin{equation} \label{def-D}
D(a) := \lim_{r \to 0} \sum_{i = 1}^n \mu_i \left( \int_{\Omega_i \setminus B_r(a_i)} \frac{e^{f_{a_i}(z)} - 1}{|z - a_i|^4} - \int_{\Bbb R^2 \setminus \Omega_i} \frac{1}{|z - a_i|^4} \right),
\end{equation}
where $f_{a_i}(z)$, $\mu_i$ are defined in (\ref{f-q}) and (\ref{mu}) respectively.
Notice that the sum in the RHS of (\ref{def-D}) can be written as
$$
\sum_{i = 1}^n \left(\int_{\Omega_i \setminus B_r(a_i)} \frac{\mu_i e^{f_{a_i}(z)}}{|z - a_i|^4} - \int_{\Bbb R^2 \setminus B_r(a_i)} \frac{\mu_i}{|z - a_i|^4} \right),
$$
where
\begin{equation}\label{kz0}
K(z) := \frac{\mu_i e^{f_{a_i}(z)}}{|z - a_i|^4} = \exp \Big(8\pi \sum_{j = 1}^n G(z, a_j) - 8\pi n G(z) \Big)
\end{equation}
is independent of $i$. Hence \eqref{def-D} of is independent of the choices of $\Omega_i$'s.

From now on, we use notation $p = \{p_1, \cdots, p_n\}$ instead of $a = \{a_1, \cdots, a_n\}$ to denote branch points.
Assume that $p = \{p_1, \cdots, p_n\} \in Y_n\setminus X_n$ is a branch point. Then $\{p_1, \cdots, p_n \} = \{-p_1, \cdots, -p_n \}$ and
\begin{equation}\label{kz}
\begin{split}
K(z) &= \exp 4\pi \Big(\sum_{j = 1}^n \big(G(z, p_j) + G(z, -p_j) - 2 G(z) \big) \Big) \\
&= e^c \prod_{i = 1}^n |\wp(z) - \wp(p_i)|^{-2}
\end{split}
\end{equation}
for some constant $c\in\mathbb{R}$. The last equality follows by the comparison of singularities. We remark here that, in comparison with \cite[\S 2]{LW4}, for non-half period points the simultaneous appearance of $\pm p_i$ is essential to arrive at the above simple looking closed form.

For convenience, we define $\Lambda_2 = \{\,i\mid p_i \in E[2]\,\}$, the two-torsion part, and for $i \not\in \Lambda_2$ we define $i^* \not \in \Lambda_2$ to be the index so that $p_{i^*} = -p_i$. 

Choose a sequence $a^k \in X_n$ with $\lim_{k \to \infty} a^k = p$. For ease of notations we drop the index $k$ and simply denote $a = (a_1, \cdots, a_n) \to (p_1, \cdots, p_n)$.

In \S\ref{geom-Xn} we show that $a \in X_n$ is equivalent to the following equation:
\begin{equation} \label{g=CP}
g_a(z) := \sum_{i = 1}^n \frac{\wp'(a_i)}{\wp(z) - \wp(a_i)}= \frac{C(a)}{\prod_{i = 1}^n (\wp(z) - \wp(a_i))}
\end{equation}
(so that ${\rm ord}_{z = 0}\,g_a(z) = 2n$) for a constant $C(a) \ne 0$ given by
$$
C(a) = \wp'(a_i) \prod_{j \ne i} (\wp(a_i) - \wp(a_j)),\quad \mbox{for any $i = 1, \ldots, n$}.
$$
For $a \in Y_n$, $C(a) = 0$ if and only if $a$ is a branch point. It is easy to describe the behavior of the limit $C(a) \to C(p) = 0$ as $a \to p$:

\begin{lemma} \label{Ca-asymp}
Let $p \in Y_n \setminus X_n$ and $a \in X_n$ near $p$. If $i \in \Lambda_2$ then
\begin{equation} \label{C-L2}
C(a) = \wp''(p_i) \prod_{j \ne i} (\wp(p_i) - \wp(p_j)) (a_i - p_i) + o(|a_i - p_i|),
\end{equation}
and if $i \not\in \Lambda_2$ then
\begin{equation} \label{C-nL2}
C(a) = \wp'(p_i)^2 \prod_{j \ne i,\,i^*} (\wp(p_i) - \wp(p_j)) (a_i + a_{i^*}) + o(|a_i + a_{i^*}|).
\end{equation}
\end{lemma}

\begin{lemma} \label{residue=0}
For $p \in Y_n$ being a branch point, the residue for
$$
P_p(z) := \prod_{i = 1}^n (\wp(z) - \wp(p_i))^{-1}
$$
at $p_i$ is zero for all $i = 1, \ldots, n$.
\end{lemma}

\begin{proof}
Choose $a \in X_n$ with $a \to p$ as above. We compute from (\ref{g=CP}) that
\begin{equation*}
\begin{split}
P_a(z) = \frac{g_a(z)}{C(a)} &= \frac{1}{C(a)} \sum_{i \in \Lambda_2} \frac{\wp'(a_i)}{\wp(z) - \wp(a_i)} \\
&\quad + \frac{1}{2C(a)} \sum_{i \not\in \Lambda_2} \Big(\frac{\wp'(a_i)}{\wp(z) - \wp(a_i)} + \frac{\wp'(a_{i^*})}{\wp(z) - \wp(a_{i^*})} \Big).
\end{split}
\end{equation*}
By Lemma \ref{Ca-asymp}, the first sum has limit
$$
\sum_{i \in \Lambda_2} \frac{\prod_{j \ne i} (\wp(p_i) - \wp(p_j))^{-1}}{\wp(z) - \wp(p_i)}
$$
when $a \to p$, which obviously has zero residue at $p_i$ because $i \in \Lambda_2$ means $p_i=\frac{1}{2}\omega_k$ in $E$ for some $k\in \{1,2,3\}$.

For the second sum, we rewrite each $i$-th summand as
$$
\frac{1}{2C} \frac{\wp'(a_i) - \wp'(-a_{i^*})}{\wp(z) - \wp(a_i)} - \frac{\wp'(a_{i^*})}{2C} \frac{\wp(a_i) - \wp(a_{i^*})}{(\wp(z) - \wp(a_i))(\wp(z) - \wp(-a_{i^*}))},
$$
which has limit
$$
\frac{1}{2} \left(\frac{\wp''(p_i)}{\wp'(p_i)^2} \frac{1}{\wp(z) - \wp(p_i)} + \frac{1}{(\wp(z) - \wp(p_i))^2}\right) \prod_{j \ne i,\,i^*} (\wp(p_i) - \wp(p_j))^{-1}.
$$
A direct Taylor expansion shows that the residues of both terms at $p_i$ ($i\notin \Lambda_2$) cancel out with each other. This proves the lemma.
\end{proof}

By Lemma \ref{residue=0}, we may rewrite
\begin{equation} \label{Pp}
P_p(z) = \prod_{i = 1}^n (\wp(z) - \wp(p_i))^{-1} = \sum_{j = 1}^n c_j \wp(z - p_j) + c_0.
\end{equation}
Since the vanishing order of the LHS at $z = 0$ is $2n$, the coefficients must satisfy the constraints
\begin{equation}\label{iii1}
\sum_{j = 1}^n c_j \wp(p_j) + c_0 = 0,
\end{equation}
\begin{equation*}
\sum_{j = 1}^n c_j \wp^{(k)}(-p_j) = 0,\quad\text{for}\;\,k = 1, \ldots, 2n - 1.\end{equation*}
Also, it is easy to see from (\ref{Pp}) that for $i \in \Lambda_2$,
\begin{equation} \label{ci-L2}
c_i = 2 \wp''(p_i)^{-1} \prod_{j \ne i} (\wp(p_i) - \wp(p_j))^{-1},
\end{equation}
and if $i \not \in \Lambda_2$ then
\begin{equation} \label{ci-nL2}
c_i = \wp'(p_i)^{-2} \prod_{j \ne i,\,i^*} (\wp(p_i) - \wp(p_j))^{-1}.
\end{equation}
In particular $c_{i^*} = c_i$.

This vector $\vec c = (c_1, \cdots, c_n)$ indeed has important geometric meaning:

\begin{lemma} \label{a'(p)}
By considering $C$ as the local holomorphic coordinate of the hyperelliptic curve $Y_n \ni a(C)$ near a branch point $p$, then we have $a'(0) = \vec{c}/2$. Moreover,
$$
\frac{\p a_j}{\p C}(0) = \frac{c_j}{2} \not\in \{0, \infty\}
$$
for $j = 1, \cdots, n$.
\end{lemma}

\begin{proof}
We first show that if $i \not\in \Lambda_2$ then
\begin{equation} \label{a=c}
\frac{\p a_i}{\p C}(0) = \frac{\p a_{i^*}}{\p C}(0).
\end{equation}
Suppose that $a(C) = (a_i(C))$ represents the point $(B, C) \in Y_n$ close to $p$, where $B = (2n - 1)\sum_{i = 1}^n \wp(a_i(C))$. Then $\tilde a(C) = (a_i(-C))$ represent the other point $(B, -C)$ with the same $B$. That is, $B = (2n - 1) \sum_{i = 1}^n \wp(a_i(-C))$ too. By the hyperelliptic structure on $Y_n$, we conclude that
$$
\{a_1(-C), \cdots, a_n(-C)\} = \{-a_1(C), \cdots, -a_n(C)\}.
$$

If $i \not\in \Lambda_2$, then we must have $a_i(-C) = -a_{i^*}(C)$ and $a_{i^*}(-C) = -a_i(C)$. Therefore, $a_i(-C) + a_{i^*}(-C) = -(a_{i}(C) + a_{i^*}(C))$ and
$$
a_i(-C) - a_{i^*}(-C) = a_i(C) - a_{i^*}(C).
$$
That is, $a_i(C) - a_{i^*}(C)$ is even in $C$, which implies (\ref{a=c}).

The lemma now follows from (\ref{C-L2})-(\ref{C-nL2}) in Lemma \ref{Ca-asymp}. For example, if $i\in\Lambda_2$, then (\ref{C-L2}) implies $\lim_{C\to 0}\frac{a_i(C)-p_i}{C}=\frac{c_i}{2}$. If $i\notin\Lambda_2$, then (\ref{a=c}) and (\ref{C-nL2}) imply
$$2\frac{\p a_i}{\p C}(0)=\frac{\p a_i}{\p C}(0)+\frac{\p a_{i^*}}{\p C}(0)=\lim_{C\to 0}\frac{a_i(C)+a_{i^*}(C)}{C}=c_i.$$
Notice that the property $c_j \ne 0, \infty$ for all $j$ is clear from the expressions in (\ref{ci-L2}) and (\ref{ci-nL2}) since (i) $p_i \not\in \Lambda$ for all $i$ and $\wp(p_i) \ne \wp(p_j)$ for all $i \ne j$, and (ii) $\wp''(p_i) \ne 0$ for $i \in \Lambda_2$ and $\wp'(p_i) \ne 0$ for $i \not\in \Lambda_2$.
\end{proof}

Using the tangent vector $\vec c$, we may derive a simple formula for $D(p)$.

\begin{theorem} \label{D(p)-formula}
Let $p \in Y_n \setminus X_n$ be a branch point of the hyperelliptic curve $Y_n$ defined by $C^2 = \ell_n(B)$. Consider the local parameter $C$ near $p$ and let $\vec c = 2a'(0) = 2\p a/\p C|_{C = 0}$. Denote also by $s = \sum_{j = 1}^n c_j$ and $c_0 = -\sum_{j = 1}^n c_j \wp(p_j)$. Then
\begin{equation} \label{D(p)-n}
\begin{split}
D(p) &= \operatorname{Im}\tau \cdot e^c \Big(|c_0 - s\eta_1|^2 + \frac{2\pi}{\operatorname{Im}\tau}{\rm Re}\,\bar s (c_0 - s\eta_1) \Big) \\&= \operatorname{Im}\tau\cdot e^c |s|^2 \Big(\Big|\frac{c_0}{s} - \eta_1\Big|^2 + \frac{2\pi}{\operatorname{Im}\tau}{\rm Re}\,\Big(\frac{c_0}{s} - \eta_1\Big) \Big).
\end{split}
\end{equation}
\end{theorem}

\begin{proof}
By Lemma \ref{a'(p)}, $\vec c$ coincides with the vector formed by the coefficients $c_1, \cdots, c_n$ appeared in the expansion formula of $P_p(z)$ in (\ref{Pp}).

Let $T \subset \Bbb R^2$ be a fundamental domain of $E_\tau$ with $p \cap \p T = \emptyset$. Then
\begin{align} \label{D(p)}
D(p) &= \lim_{r \to 0} \left( e^c \int_{T \setminus \cup_i B_r(p_i)} |P_p(z)|^2 - \sum_{i = 1}^n \int_{\Bbb R^2 \setminus B_r(p_i)}\frac{\mu_i}{|z - p_i|^4} \right)\\
&=\lim_{r\to 0}\left(  e^{c}\int_{T\setminus \cup_{i=1}%
^{n}B_{r}(p_{i})}|P_{p}(z)|^{2}-\sum_{i=1}^{n}\frac{\pi \mu_{i}}{r^{2}}\right).\nonumber
\end{align}

Consider an anti-derivative of $P_p(z)$:
\begin{equation}\label{lp}
L_p(z) := \int_0^z P_p(w)\,dw = -\sum_{j = 1}^n c_j \zeta(z - p_j) + c_0 z.
\end{equation}
For $i = 1, 2$, we define the ``quasi-periods'' $\chi_i$ by
\begin{equation}\label{chi-i}
\begin{split}
\chi_i = L_p(z + \omega_i) - L_p(z) = c_0 \omega_i - s\eta_i.
\end{split}
\end{equation}

To compute $D(p)$, we note from (\ref{Pp}) that
\[
P_{p}(z)=\frac{c_{i}}{(z-p_{i})^{2}}+O(1)
\]
and from (\ref{kz0}), the definition (\ref{mu}) of $\mu_i$ that
\[
K(z)=\frac{\mu_{i}}{|z-p_{i}|^{4}}+O(|z-p_{i}|^{-2}).
\]
Inserting these into (\ref{kz}) leads to%
\begin{equation}\label{mu-ci}
\mu_{i}=e^{c}|c_{i}|^{2},\quad\forall 1\leq i\leq n.
\end{equation}
Now we denote%
\[
L_{p}(z)=u+\sqrt{-1}v,\quad z=x+\sqrt{-1}y.
\]
Then $P_{p}(z)=L_{p}^{\prime}(z)=u_{x}-iu_{y}$, i.e.%
\[
|P_{p}(z)|^{2}=u_{x}^{2}+u_{y}^{2}=(uu_{x})_{x}+(uu_{y})_{y}\text{ for
}z\text{ outside }\{p_{1},\cdot \cdot \cdot,p_{n}\},
\]
so%
\begin{align*}
&  \int_{T\setminus \cup_{i=1}^{n}B_{r}(p_{i})}|P_{p}(z)|^{2}\\
=&\int_{\partial T}\big(uu_{x}dy-uu_{y}dx\big)-\sum_{i=1}^{n}%
\int_{|z-p_{i}|=r}\big(uu_{x}dy-uu_{y}dx\big).
\end{align*}
Applying (\ref{chi-i}) we obtain
\begin{align}\label{inte}
\int_{\partial T}\big(uu_{x}dy-uu_{y}dx\big) &  =\int_{\partial
T}udv=-\frac{1}{2}\operatorname{Im}\int_{\partial T}L_{p}%
d\bar{L}_{p}\\
&  =\frac{1}{2}\operatorname{Im}(\bar{\chi}_{1}\chi_{2}-\chi_{1}\bar{\chi}%
_{2}).\nonumber
\end{align}
Since near $p_{i}$,%
\[
u+\sqrt{-1}v=L_{p}(z)=-\frac{c_{i}}{z-p_{i}}+f(z),
\]
where $f(z)$ is holomorphic in a neighborhood of $p_{i}$, it is easy to prove
that%
\[
-\int_{|z-p_{i}|=r}\big(uu_{x}dy-uu_{y}dx\big)=-\int_{|z-p_{i}|=r}%
udv=\frac{\pi|c_{i}|^{2}}{r^{2}}+O(r).
\]
Therefore, we conclude from (\ref{mu-ci}) and (\ref{inte}) that
\begin{align}\label{regulator}
D(p) &  =\lim_{r\rightarrow0}\left(  e^{c}\int_{T\setminus \cup
_{i=1}^{n}B_{r}(p_{i})}|P_{p}(z)|^{2}-\sum_{i=1}^{n}\frac{\pi \mu_{i}}{r^{2}%
}\right)  \\
&  =\frac{e^{c}}{2}\operatorname{Im}(\bar{\chi}_{1}\chi_{2}-\chi_{1}\bar{\chi
}_{2}).\nonumber
\end{align}

By direct substitution, we compute
\begin{equation*}
\begin{split}
\bar\chi_1 \chi_2 - \chi_1 \bar \chi_2 &= |c_0|^2 (\bar\omega_1 \omega_2 - \omega_1 \bar \omega_2) + |s|^2 (\bar\eta_1 \eta_2 - \eta_1 \bar\eta_2) \\
&\quad + \bar c_0 s(\eta_1 \bar\omega_2 - \eta_2 \bar\omega_1) - c_0 \bar s(\bar \eta_1 \omega_2 - \bar \eta_2 \omega_1).
\end{split}
\end{equation*}
Now we plug in $\omega_1 = 1$, $\omega_2 = \tau = a + bi$, and use the Legendre relation $\eta_1 \omega_2 - \eta_2 \omega_1 = 2\pi i$. Then
\begin{equation*}
\begin{split}
\bar\omega_1 \omega_2 - \omega_1 \bar \omega_2 &= 2ib, \\
\bar\eta_1 \eta_2 - \eta_1 \bar\eta_2 &= 2i(b|\eta_1|^2 - \pi(\eta_1 + \bar\eta_1)), \\
\eta_1 \bar\omega_2 - \eta_2 \bar\omega_1 &= 2i(\pi - \eta_1 b).
\end{split}
\end{equation*}
Hence
\begin{equation*}
\begin{split}
D(p) &= e^c \Big(b(|c_0|^2 + |s\eta_1|^2) + 2 {\rm Re}\big(c_0 \bar s (\pi - \bar\eta_1 b) - |s|^2 \pi \eta_1 \big)\Big)\\
&= b e^c \Big(|c_0 - s\eta_1|^2 + \frac{2\pi}{b}{\rm Re}\,\bar s (c_0 - s\eta_1) \Big).
\end{split}
\end{equation*}
This proves the theorem.
\end{proof}

In fact there is a simple geometric interpretation of the expression appeared in the RHS of (\ref{D(p)-n}).

\begin{proposition} \label{Jacobian}
Consider the vector-valued map $(E^\times)^n \to \Bbb R^2$ defined by
$$
a \mapsto \phi(a):= -4\pi \sum_{i = 1}^n \nabla G(a_i).
$$
Let $C = u + iv \mapsto a(C) \in E^n$ be a local holomorphic parametrization of a Riemann surface $V \subset E^n$. Then the Jacobian $J(\phi\circ a)(u, v)$ is given by
\begin{equation}
\det \Big( \frac{\p\phi}{\p u}, \frac{\p\phi}{\p v}\Big) = -\Big(|c_0 - s\eta_1|^2 + \frac{2\pi}{\operatorname{Im}\tau} {\rm Re}\,\bar s(c_0 - s\eta_1)\Big),
\end{equation}
where $\vec c = (c_i) := 2a'(C)$, $s := \sum_{i = 1}^n c_i$, and $c_0 := -\sum_{i = 1}^n c_i \wp(a_i)$.
\end{proposition}

\begin{proof}
Denote $a_j = x_j + \sqrt{-1} y_i$, $b=\operatorname{Im}\tau$ and $\phi = (\phi_1, \phi_2)^T$. By (\ref{G_z}), we have
\begin{equation} \label{grad-phi}
\begin{split}
\phi_1 &= 2\,{\rm Re}\,(\sum_i \zeta(a_i) - \eta_1 a_i),\\
\phi_2 &= -2\,{\rm Im}\,(\sum_i \zeta(a_i) - \eta_1 a_i) - \frac{4\pi}{b} \sum y_i.
\end{split}
\end{equation}
The chain rule shows that
\begin{equation*}
\begin{split}
\p_u \phi_1 &= -2\,{\rm Re}\Big[\sum_i (\wp(a_i) + \eta_1)\frac{\p a_i}{\p C}\Big] = {\rm Re}\,(c_0 - s\eta_1),\\
\p_v \phi_1 &= -2\,{\rm Re}\Big[\sum_i (\wp(a_i) + \eta_1) \frac{\p a_i}{\p C}\sqrt{-1}\Big] = -{\rm Im}\,(c_0 - s\eta_1),\\
\p_u \phi_2 &= 2\,{\rm Im}\Big[\sum_i (\wp(a_i) + \eta_1) \frac{\p a_i}{\p C} \Big] - \frac{4\pi}{b} \sum_i \frac{\p y_i}{\p u} \\
&= -{\rm Im}\,(c_0 - s\eta_1) - \frac{2\pi}{b}{\rm Im}\,s, \\
\p_v \phi_2 &= 2\,{\rm Im}\Big[\sum_i (\wp(a_i) + \eta_1) \frac{\p a_i}{\p C} \sqrt{-1}\Big] - \frac{4\pi}{b} \sum_i \frac{\p y_i}{\p v} \\
&= -{\rm Re}\,(c_0 - s\eta_1) - \frac{2\pi}{b} {\rm Re}\,s.
\end{split}
\end{equation*}
Hence the Jacobian is given by
\begin{equation*}
\begin{split}
&-|c_0 - s\eta_1|^2 - \frac{2\pi}{b} ({\rm Re}\,(c_0 - s\eta_1)\, {\rm Re}\,s + {\rm Im}\,(c_0 - s\eta_1)\, {\rm Im}\,s) \\
&= - \Big(|c_0 - s\eta_1|^2 + \frac{2\pi}{b} {\rm Re}\,\bar s(c_0 - s\eta_1)\Big)
\end{split}
\end{equation*}
as expected.
\end{proof}

\begin{corollary} \label{D=Jac}
For $p \in Y_n \setminus X_n$ with local coordinate $C$, we have
\begin{equation} \label{D=J}
D(p) = -{\rm Im}\tau\, e^c J(\phi\circ a)(0)
\end{equation}
for some constant $c$.
\end{corollary}

\begin{proof}
This follows from Theorem \ref{D(p)-formula} and Proposition \ref{Jacobian}.
\end{proof}

Corollary \ref{D=Jac} will play important role in our subsequent degeneration analysis of these branch points $p \in Y_n \setminus X_n$. One may also interpret the above proof of it as a stationary phase integral calculation.

\begin{example}
For $n = 1$, $c_0 = -c_1\wp(p_1) = -c_1 e_i$ if $p_1 = \tfrac{1}{2} \omega_i$, and $s = c_1$. The formula reduces to the one for $\det D^2 G(p)$ first studied in \cite{LW}:
$$
|e_i + \eta_1|^2 - \frac{2\pi}{{\rm Im}\tau}{\rm Re}(e_i + \eta_1).
$$
\end{example}

\section{Proof of Theorem \ref{thm1}: Analytic adjunction} \label{s:proof}

It is elementary to see that for $\chi_1 = a_1 + b_1i$ and $\chi_2 = a_2 + b_2 i$,
$$
\bar\chi_1 \chi_2 - \chi_1 \bar \chi_2 = 2i(a_1b_2 - a_2 b_1).
$$
Hence the formula in (\ref{regulator}) says that $D(p)$ is exactly $e^c$ times the signed area spanned by $\chi_1$ and $\chi_2$ in $\Bbb R^2$. Indeed, $\chi_1 = c_0 - s\eta_1 = -\sum_{j = 1}^n c_j (\wp(p_j) + \eta_1)$. So we may rewrite (\ref{D(p)-n}) as
\begin{equation} \label{D(p)-det}
D(p) = -{\rm Im}\tau\, e^c |s|^2 \begin{vmatrix} -{\rm Re}\,\chi_1 s^{-1} & +{\rm Im}\,\chi_1 s^{-1} \\ +{\rm Im}\,\chi_1 s^{-1} & \displaystyle {\rm Re}\,\chi_1 s^{-1} + \frac{2\pi}{{\rm Im}\tau}\end{vmatrix}.
\end{equation}

Formula (\ref{D(p)-det}) suggests the possibility for interpreting $D(p)$ in terms of the determinant of the Hessian of some ``Green function'' for general $n \in \Bbb N$. To find such a Green function on $\bar X_n$ will require the search for a suitable conformal metric on it. Alternatively we consider the multiple Green function $G_n$ defined in (\ref{multiple-G}):
$$
G_n(z_1, \cdots, z_n) := \sum_{i < j} G(z_i - z_j) - n \sum_{i = 1}^n G(z_i)
$$
for $z = (z_1, \cdots, z_n) \in (E^\times)^n \setminus \Delta_n$.
Then
$G_n$ is a Green function on $E^n$ with divisor $D_n$ where $(E^\times)^n \setminus \Delta_n = E^n \setminus D_n$. Recall that \emph{$p$ is a branch point of $Y_n$ if and only if it is a trivial critical point of $G_n$}.

\begin{theorem} [Analytic adjunction formula] \label{DD-conj}
For any fixed $n \in \Bbb N$ and any branch point $p = (p_1, \cdots, p_n) \in Y_n$, there is a constant $c_p \ge 0$ such that
\begin{equation*}
\det D^2 G_n(p) = (-1)^n c_p D(p).
\end{equation*}
Moreover, $c_p = 0$ precisely when the associated hyperelliptic curve $Y_n(\tau)$ for $E = E_\tau$ is singular at $p$. There are only finitely many such tori $E_\tau$ for each $n$.
\end{theorem}

For $n = 1$, this is \cite[Theorem 0.4]{LW4}. For $n = 2$, a direct check based on Theorem \ref{D(p)-formula} is still possible (c.f.~Example \ref{DD2}). For $n \ge 3$ the $D^2 G_n$ is a $2n \times 2n$ matrix and it is cumbersome to compute $\det D^2 G_n(p)$ directly. The proof of Theorem \ref{DD-conj} given below is based on Corollary \ref{D=Jac}.

\begin{proof}
It was proved in \cite[\S 5.3]{CLW} (recalled in (\ref{ff3})--(\ref{ff4})) that the system of equations (\ref{poho}) given by $-2\pi \nabla G_n(a) = 0$ is equivalent to holomorphic equations  $g^1(a) =\cdots = g^{n - 1}(a) = 0$ with
\begin{equation} \label{zeta-eqn}
g^i(a) = \sum_{j\ne i}^n (\zeta(a_i - a_j) + \zeta(a_j) - \zeta(a_i)),\quad 1\le i\le n-1,
\end{equation}
which defines $Y_n$, and the non-holomorphic equation $g^n(a) = 0$ with
\begin{equation}\label{g-n}
g^n(a) = \tfrac{1}{2}\phi(a) = -2\pi \sum_{i = 1}^n \nabla G(a_i).
\end{equation}
By (\ref{G_z}), we easily obtain for $1\le i\le n-1$,
\begin{equation}\label{gia}
g^{i}(a)=-2\pi\left(\sum_{j\neq i}2G_z(a_i-a_j)-2nG_z(a_i)+\sum_{j=1}^n 2G_z(a_j)\right).
\end{equation}
For any $i$, we have
$$
\nabla_i G_n(a)=\sum_{j\neq i}\nabla G(a_i-a_j)-n\nabla G(a_i).
$$
By taking into account that $\nabla G \mapsto 2G_z$ has matrix $\begin{pmatrix} 1 & 0 \\ 0 & -1 \end{pmatrix}$ and $g^n = \frac{2\pi}{n}\sum_{i = 1}^n \nabla_{i} G_n$, the equivalence between the map $a \mapsto g(a) := (g^1(a), \cdots,$ $g^n(a))^T$ and $-2\pi\nabla G_n$ is induced by a real $2n \times 2n$ matrix $A$ given by
\begin{equation*}
A= \begin{bmatrix}
1 & & & & & 1 &  \\
& 1 & & & & & -1 \\
& & \ddots & & & \vdots & \vdots\\
& & & 1& & 1 &  \\
& & & & 1& & -1 \\
& & & & & 1 &  \\
& & & & & & 1
\end{bmatrix}
\cdot
\begin{bmatrix}
1 & & & & &  &  \\
& -1 & & & & &  \\
& & \ddots & & & & \\
& & & 1& &  &  \\
& & & & -1& &  \\
\tfrac{-1}{n}& & \cdots& \tfrac{-1}{n}& & \tfrac{-1}{n} &  \\
&\tfrac{-1}{n} & \cdots& & \tfrac{-1}{n}& & \tfrac{-1}{n}
\end{bmatrix}.
\end{equation*}
In other words, by considering
$g^{k}=({\rm Re}g^{k}, {\rm Im}g^{k})^T$ for $1\le k\le n-1$ and $G_z=({\rm Re}G_z, {\rm Im}G_z)^T$, we have $2G_z=\begin{pmatrix} 1 & 0 \\ 0 & -1 \end{pmatrix}\nabla G$. Inserting this into (\ref{gia}), it is easy to obtain $g(a)=-2\pi A \nabla G_n(a)$.
Consequently,
\begin{equation} \label{J(g)}
J(g)(a) = J(-2\pi A\nabla G_n)(a) = \frac{(-1)^{n - 1}}{n^2} (2\pi)^{2n}\det D^2 G_n(a),
\end{equation}
so it suffices to compute (the real Jacobian) $J(g)$.

Let $p \in Y_n \setminus X_n$ and consider a holomorphic parametrization $C \mapsto a(C)$ of a branch of $Y_n$ near $p$, where $p$ corresponds to $C = 0$. Notice that if $p$ is not a singular point of $Y_n$, i.e.~$B_p$ is a simple zero of $\ell_n(B) = 0$, then there is only one branch of $Y_n$ near $p$ and the map $C \to a(C)$ is unique.

We denote
$$
C = u + \sqrt{-1} v, \quad a_k = x^k + \sqrt{-1} y^k, \quad g^k = U^k + \sqrt{-1} V^k, \quad 1 \le k \le n.
$$
Along $Y_n$ we have by chain rule (denote $g^i_j = \p g^i/\p a_j$)
\begin{equation} \label{der-level}
0 = \frac{\p g^i}{\p C} = \sum_{j = 1}^n g^i_j\, \frac{\p a_j}{\p C}, \qquad 1 \le i \le n - 1.
\end{equation}
If $g^n$ is also holomorphic, then (\ref{der-level}) can be used to evaluate the ``complex determinant'' $\det D^{\Bbb C}g = \det (g^i_j)$ by elementary column operations. For example, if $\p a_k/\p C \ne 0$ then we may eliminate all the entries of the $k$-th column except the last ($n$-th) one. The case $k = n$ reads as:
\begin{equation} \label{adjunction}
\det D^{\Bbb C} g = \det(g^i_j)_{i, j = 1}^{n - 1} \times \frac{\p g^n}{\p C}\times\Big(\frac{\p a_n}{\p C}\Big)^{-1}.
\end{equation}

In the current case $g^n = \tfrac{1}{2}\phi$ is not holomorphic (see (\ref{grad-phi}) for the additional linear term $-2\pi \sum_k y^k/b$ for $V^n = \tfrac{1}{2}\phi_2$). The same argument via implicit functions still applies if we work with the real components $U^k, V^k$ and real variables $x^k, y^k$ and $u, v$ instead.

More precisely, (\ref{der-level}) takes the real form: For $1 \le i \le n - 1$,
\begin{equation} \label{der-level2}
\begin{split}
0 = \begin{bmatrix} U^i_u & U^i_v \\ V^i_u & V^i_v \end{bmatrix}
= \sum_{k = 1}^n
\begin{bmatrix} U^i_{x^k} & U^i_{y^k} \\ V^i_{x^k} & V^i_{y^k} \end{bmatrix} \begin{bmatrix} x^k_u & x^k_v \\ y^k_u & y^k_v \end{bmatrix}.
\end{split}
\end{equation}
The two rows are equivalent by the Cauchy--Riemann equation.

The elementary column operation on the $2n \times 2n$ real jacobian matrix $Dg$ is now replaced by the right multiplication with the matrix
\begin{equation*}
R_n := \begin{bmatrix}
1 & & & x^1_u & x^1_v \\
& 1 & & y^1_u & y^1_v \\
& & \ddots & \vdots & \vdots\\
& & & x^n_u & x^n_v \\
& & & y^n_u & y^n_v
\end{bmatrix}.
\end{equation*}
In fact we may do so for any $(2k - 1, 2k)$-th pair of columns---since
$$
\begin{vmatrix} x^k_u & x^k_v \\ y^k_u & y^k_v \end{vmatrix} = |a'_k(0)|^2 \ne 0, \infty
$$
by Lemma \ref{a'(p)}, and get a similar matrix $R_k$. We take $R = R_n$ below.

Denote by $D'g$ the principal $2(n - 1) \times 2(n - 1)$ sub-matrix of $Dg$. Notice from (\ref{g-n}) that
$$
\tfrac{1}{2} D(\phi\circ a) = \begin{bmatrix} U^n_u & U^n_v \\ V^n_u & V^n_v \end{bmatrix}
= \sum_{k = 1}^n
\begin{bmatrix} U^n_{x^k} & U^n_{y^k} \\ V^n_{x^k} & V^n_{y^k} \end{bmatrix} \begin{bmatrix} x^k_u & x^k_v \\ y^k_u & y^k_v \end{bmatrix},
$$
which is precisely the right bottom $2 \times 2$ sub-matrix of $(Dg)R$. Hence it follows from (\ref{der-level2}) that
$$
(Dg)R=\begin{bmatrix} D'g & 0 \\ * & \tfrac{1}{2} D(\phi\circ a) \end{bmatrix},
$$
which can be used to calculate the determinant:
$$
\det Dg\, \det R = \det ((Dg)R) = \det D' g\, \det \tfrac{1}{2} D(\phi\circ a).
$$
By (\ref{J(g)}) and Corollary \ref{D=Jac}, we get
\begin{equation*}
\det D^2 G_n(p) = \frac{(-1)^n n^2 e^{-c}}{4{\rm Im}\,\tau\, (2\pi)^{2n}} \frac{|\det D'^{\Bbb C} g(p)|^2}{|a_n'(0)|^2} \,D(p),
\end{equation*}
where $D'^{\Bbb C} g$ is the principal $(n - 1) \times (n - 1)$ sub-matrix of $D^{\Bbb C}g$. Here we used $\det D' g(p)=|\det D'^{\Bbb C} g(p)|^2$ because $g^k$ is holomorphic for any $1\le k\le n-1$. Thus
\begin{equation} \label{cp-expression}
c_p = \frac{n^2 e^{-c}}{4{\rm Im}\,\tau\, (2\pi)^{2n}} \frac{|\det D'^{\Bbb C} g(p)|^2}{|a_n'(0)|^2} \ge 0.
\end{equation}

To complete the proof of Theorem \ref{DD-conj}, we recall the standard Jacobian criterion for smoothness of the point $p \in Y_n$. Since $g^1 = 0, \cdots, g^{n - 1} = 0$ are the defining equations for $Y_n$, $p \in Y_n$ is a non-singular point if and only if there is some $(n - 1) \times (n  - 1)$ minor of the $(n - 1) \times n$ matrix $D^{\mathbb{C}}\tilde{g}(p)$ which does not vanish at $p$, where $\tilde{g}:=(g^1,\cdots,g^{n-1})^T$. Notice that (\ref{cp-expression}) is valid for all choices of those minors (with $a'_n(0)$ being replaced by $a'_k(0)$), thus $p \in Y_n$ is non-singular is indeed equivalent to $\det D'^{\Bbb C} g(p) \ne 0$ (which actually implies that any $(n - 1) \times (n  - 1)$ minor does not vanish at $p$). Since $p \in Y_n \setminus X_n$ is a branch point and $Y_n$ is defined by the hyperelliptic equation $C^2 = \ell_n(B)$, this is precisely the case when $B_p$ is a simple zero of $\ell_n(B) = 0$. The proof is complete.
\end{proof}

\begin{example} [The case $n = 2$] \label{DD2}
For any flat torus $E_\tau$ and $p \in Y_2(\tau) \setminus X_2(\tau)$, we compute directly the constant $c_p = c_p(\tau) \ge 0$ such that
$$
\det D^2 G_2(p) = c_p D(p).
$$
To serve as a consistency check with (\ref{cp-expression}) we will not follow the procedure used in the proof of Theorem \ref{DD-conj}. Instead we will compute $\det D^2 G_2(p)$ directly. It will be clear that $c_p(\tau) > 0$ if $\tau \not\equiv e^{\pi/3}$ under the ${\rm SL}(2, \Bbb Z)$ action.

By Example \ref{lame-curve} (2), we see that the five branch points in $Y_n$ are given by $\{(\tfrac{1}{2} \omega_i, \tfrac{1}{2} \omega_j) \mid i \ne j\}$ and $\{(q_{\pm}, -q_{\pm}) \mid \wp(q_{\pm}) = \pm \sqrt{g_2/12}\}$. Note that $\wp^2(q) = \tfrac{1}{12} g_2$ if and only if $\wp''(q) = 0$. The only case that these five points reduce to four points is when $g_2 = 0$. This happens precisely when $\tau \equiv e^{\pi/3}$ and then $\bar Y_n$ becomes a singular hyperelliptic curve.

To compute the Hessian of $G_2$, we recall the formulae \cite[(2.4) and (2.5)]{LW4} for the second partial derivatives of $G$. Namely,
\begin{equation} \label{Hessian}
\begin{split}
\det D^2 G = \frac{-1}{4\pi^2}\Big(|(\log \T)_{zz}|^2 + \frac{2\pi}{b} {\rm Re}\,(\log \T)_{zz}\Big),
\end{split}
\end{equation}
where $b={\rm Im}\,\tau$ and in terms of the Weierstrass theory
\begin{equation} \label{period-1/2}
(\log \T)_{zz} (\tfrac{1}{2} \omega_i) = -\wp(\tfrac{1}{2} \omega_i) - \eta_1 = -(e_i + \eta_1),
\end{equation}
where $(\log \T)_z (z) = \zeta(z) - \eta_1 z$ is used.

First we compute $D^2 G(p)$ for $p = (\tfrac{1}{2} \omega_i, \tfrac{1}{2} \omega_j)$. Denote by $\tfrac{1}{2}\omega_k$ the third remaining half period point. Notice that $\wp(\tfrac{1}{2} \omega_i - \tfrac{1}{2} \omega_j) = \wp(\tfrac{1}{2}\omega_k) = e_k$. For simplicity we write
$$
w_k := (\log\T)_{zz}(\tfrac{1}{2} \omega_k) = -(e_k + \eta_1) = u_k + v_k i,
$$
and similarly for the indices $i, j$. Then we have
\begin{equation*}
D^2G_2(p) = \frac{1}{2\pi}
\begin{pmatrix}
-u_k + 2u_i & v_k - 2v_i & u_k & -v_k \\
v_k - 2 v_i & u_k - 2 u_i - \frac{2\pi}{b} & -v_k & -u_k - \tfrac{2\pi}{b} \\
u_k & -v_k & -u_k + 2u_j & v_k - 2v_j \\
-v_k & -u_k - \tfrac{2\pi}{b} & v_k - 2v_j & u_k - 2u_j - \tfrac{2\pi}{b}
\end{pmatrix}.
\end{equation*}
A lengthy yet straightforward calculation shows that
\begin{multline} \label{det2}
\det D^2 G_2(p) = \\
\frac{4}{(2\pi)^4} \Big(|2e_i e_j + e_k^2 - 3e_k \eta_1|^2 + \frac{2\pi}{b} {\rm Re}\,\big(3\bar e_k(2e_i e_j + e_k^2 - 3e_k \eta_1)\big)\Big).
\end{multline}
The details will be omitted here. We only note that when $\tau \in i \Bbb R$, all $e_l$'s and $\eta_1$ are real numbers. Thus all the imaginary parts vanish: $v_1 = v_2 = v_3 = 0$. In this case (\ref{det2}) can be verified easily.

By (\ref{ci-L2}) and the fact that $\wp''(\tfrac{1}{2}\omega_i) = 2(e_i - e_j)(e_i - e_k)$, we compute
\begin{equation*}
\begin{split}
c_1 &= 2 \wp''(\tfrac{1}{2}\omega_i)^{-1} (e_i - e_j)^{-1} = (e_i - e_j)^{-2} (e_i - e_k)^{-1}, \\
c_2 &= (e_j - e_i)^{-2} (e_j - e_k)^{-1}, \\
s &= c_1 + c_2 = -3e_k (e_i - e_j)^{-2} (e_i - e_k)^{-1} (e_j - e_k)^{-1},\\
c_0 &= -(c_1 e_i + c_2 e_j) = - (2e_i e_j + e_k^2)(e_i - e_j)^{-2}(e_i - e_k)^{-1}(e_j - e_k)^{-1}.
\end{split}
\end{equation*}
By Theorem \ref{D(p)-formula}, we get
$$
D(p) = c(p) \Big(|2e_i e_j + e_k^2 - 3e_k \eta_1|^2 + \frac{2\pi}{b} {\rm Re}\,\big(3\bar e_k(2e_i e_j + e_k^2 - 3e_k \eta_1)\big)\Big)
$$
with $c(p) = b e^c |e_i - e_j|^{-2} |e_i - e_k|^{-1} |e_j - e_k|^{-1}$. Thus
$$
\det D^2 G_2(\tfrac{1}{2} \omega_i, \tfrac{1}{2} \omega_j) = c_p D(\tfrac{1}{2} \omega_i, \tfrac{1}{2} \omega_j)
$$
with
\begin{equation} \label{cp1}
c_p = (4 \pi^4 c(p))^{-1} = e^{-c}|e_i - e_j| |e_i - e_k| |e_j - e_k|/(4b\pi^4) > 0.
\end{equation}

Next we consider $p = (q, -q)$ with $q \in \{q_+, q_-\}$. Let $\mu = \wp(q)$. Since $\wp''(q) = 0$, we have also $\wp(2q) = -2\wp(q) = -2\mu$ by the addition (duplication) formula. Denote by $\mu = u + iv$ and $\eta_1 = s + it$. Then we have
\begin{equation*}
D^2G_2(p) =
\frac{1}{2\pi}
\begin{pmatrix}
-4u - s & 4v + t & 2u - s & -2v + t \\
4v + t & 4u + s - \tfrac{2\pi}{b} & -2v + t & -2u + s - \tfrac{2\pi}{b} \\
2u - s & -2v + t & -4u - s & 4v + t \\
-2v + t & -2u + s - \tfrac{2\pi}{b} & 4v + t & 4u + s - \tfrac{2\pi}{b}
\end{pmatrix}.
\end{equation*}
A straightforward calculation easier than the previous case shows that the determinant is given by
\begin{equation} \label{det2-2}
\begin{split}
\det D^2G_2(p) &= \frac{144}{(2\pi)^4}(u^2 + v^2) \Big((u + s)^2 + (v + t)^2 - \frac{2\pi}{b}(u + s)\Big) \\
&= \frac{9}{\pi^4} |\wp(q)|^2 \Big(|\wp(q) + \eta_1|^2 - \frac{2\pi}{b} {\rm Re}(\wp(q) + \eta_1)\Big).
\end{split}
\end{equation}

By (\ref{ci-nL2}), we compute easily that $c_1 = c_2 = \wp'(q)^{-2}$, $c_0 = -2 \wp(q) \wp'(q)^{-2}$, and $s = c_1 + c_2 = 2 \wp'(q)^{-2}$. Hence by Theorem \ref{D(p)-formula}
\begin{equation*}
\begin{split}
D(p) &= 4b e^c |\wp'(q)|^{-4} \Big(|-\wp(q) - \eta_1|^2 + \frac{2\pi}{b}{\rm Re}(-\wp(q) - \eta_1)\Big) \\
&= c_p^{-1} \det D^2 G_2(p),
\end{split}
\end{equation*}
where
\begin{equation} \label{cp2}
c_p = \frac{9 e^{-c}}{4b \pi^4} |\wp(q)|^2 |\wp'(q)|^4 \ge 0.
\end{equation}
Since $\wp'(q) \ne 0$, $c_p > 0$ unless $\wp(q) = \pm \sqrt{g_2/12}= 0$. This is the case precisely when $\tau$ is equivalent to $e^{\pi i/3}$. We leave the simple consistency check of (\ref{cp1}) and (\ref{cp2}) with the general formula (\ref{cp-expression}) to the readers.
\end{example}

%

\end{document}